\documentclass[12pt]{article}

\usepackage{amsmath,amssymb,amsthm,xspace,amscd}

\theoremstyle{remark}
\newtheorem{remark}{Remark}

\theoremstyle{plain}
\newtheorem{theorem}{Theorem}

\newtheorem{proposition}{Proposition}
\newtheorem{lemma}{Lemma}

\newcommand{\Id}{{\mathbb I\rm d}}
\newcommand{\R}{{\mathbb R}}
\newcommand{\Z}{{\mathbb Z}}
\newcommand{\unitcircle}{{\mathbb T}^1}
\newcommand{\Pdyn}{{\mathcal P}}

\newcommand{\eps}{{\varepsilon}}

\newcommand{\cf}{{\check f}}
\newcommand{\cg}{{\check g}}

\begin{document}

\title{A circle diffeomorphism with breaks that is smoothly linearizable}
\author{Alexey Teplinsky
\thanks{Institute of Mathematics, National Academy of Sciences of Ukraine, 3 Tereshchenkivska Street, Kiev, Ukraine 01601; e-mail:\ teplinsky@imath.kiev.ua, teplinsky.a@gmail.com}}
\maketitle
\begin{abstract}
In this paper we answer positively a question of whether it is possible for a circle diffeomorphism with breaks to be smoothly conjugate to a rigid rotation in the case when its breaks are lying on pairwise distinct trajectories. An example constructed is a piecewise linear circle homeomorphism that has four break points lying on distinct trajectories, and whose invariant measure is absolutely continuous w.~r.~t.\ the Lebesgue measure. The irrational rotation number for our example can be chosen a Roth number, but not of bounded type.
\end{abstract}

\section{Introduction}

Any orientation-preserving homeomorphism $T$ of the unit circle $\unitcircle=\R/\Z$ with an irrational rotation number $\rho=\rho(T)$ is uniquely ergodic~\cite{KH-book}, i.~e.\ it possesses a unique invariant probability measure $\mu=\mu(T)$. If $T$ is piecewise $C^1$, and the total variation of $\log T'$ is finite, then by the renown Denjoy's theorem it is topologically conjugate to the rigid (linear) rotation $R_\rho:\xi\mapsto\xi+\rho$, i.~e.\ there exists a homeomorphism $\phi$ (which is unique up to an additive constant) of $\unitcircle$ such that
\begin{equation}\label{eq:conjugacy}
\phi\circ T\circ\phi^{-1}=R_\rho.
\end{equation}
The invariant measure $\mu$ and the conjugacy (i.~e., the linearizing change of variables) $\phi$ are related by the equality $\phi(\xi_1)-\phi(\xi_0)=\int_{\xi_0}^{\xi_1}d\mu(\xi)$ for any arc $[\xi_0,\xi_1]\subset\unitcircle$. The question of how smooth this conjugacy $\phi$ is, depending on smoothness of $T$, has been studied extensively since 1960's.

The case when $T$ is a diffeomorphism, is covered by Herman's theory and its later developments~\cite{Herman,KO-etds1989,KT-invent2009}.

Quite different is the case when $T$ is a diffeomorphism with breaks, i.~e.\ it is piecewise $C^1$, but not $C^1$ (and we still assume that the total variation of $\log T'$ is finite). Let us call the `size of the break' the ratio of derivatives at the break point from the left and from the right. Since the first paper~\cite{DK-fa1998}, in which it was shown that the invariant measure of a circle diffeomorphism with a single break of non-unit size is singular (with respect to the Lebesgue measure $\ell$), there appeared a number of publications proving the singularity of the invariant measure in different cases. Recently, two groups of scientists have independently proved~\cite{DMS-iran2012,AM-etds2014} the following most general result: if the product of sizes of all breaks is non-unit, then the invariant measure is singular.

The case of unit product of sizes of breaks appeared to be difficult to investigate. Let us notice at this point that several breaks lying on the same trajectory are in a sense equivalent to one single break of a size that is the product of their sizes, therefore without lost of generality one may consider only diffeomorphisms with breaks that lie on pairwise distinct trajectories. It is shown in~\cite{Herman} (although not formulated as a statement) that a piecewise linear diffeomorphism with two breaks (the product of their sizes is unit automatically) has singular invariant measure unless the break points lie on the same trajectory. In~\cite{DL-2006,AM-etds2014}, it was proved that for a diffeomorphism with multiple breaks lying on distinct trajectories, the invariant measure is singular provided that rotation number is of bounded type (i.~e., its continued fraction expansion is a bounded sequence). Generally, there was a colloquial belief among specialists on the subject that the invariant measure is singular w.~r.~t. $\ell$ if at least for one trajectory the product of sizes of all breaks it contains is non-unit.

In this paper, we disprove that hypothesis presenting an example of piecewise linear homeomorphism of the unit circle that has four breaks lying on distinct trajectories, but nevertheless is absolutely continuously linearizable. Its rotation number can be chosen a Roth number (those are the numbers Diophantine with any positive Diophantine exponent), but not of bounded type. The linearizing conjugacy is {\em essentially} absolutely continuous in the sense that it is not piecewise $C^1$, possessing an everywhere dense set of points where it is not differentiable.

The example we construct is also interesting in a framework of rigidity theory for circle diffeomorphisms with breaks~\cite{KT-cmp2013}, which asks a more general question of whether given two such maps are smoothly conjugate. It was proved in~\cite{CS-II} that in the case of break equivalence of the two maps (a property assuming that a conjugacy sends each break point of the first map to a break point of the second one with the same size of break), unit product of sizes of breaks and bounded combinatorics (this notion generalizes bounded type of rotation number), that conjugacy is $C^1$. On the other hand, very recently it was shown~\cite{A-manuscript,DMS-arxiv} that if two circle diffeomorphisms with breaks with the same irrational rotation number have different products of sizes of breaks, then every conjugacy between them is a singular function. Our example suggests that in the case of the equal products, the condition of generalized break equivalence (i.~e., a break equivalence adjusted by a piecewise smooth change of coordinates) is not necessary for absolute continuity of the conjugacy.

In Section~2, we define our example for given sizes of breaks and rotation number and formulate the theorem about absolute continuity of its invariant measure w.~r.~t. $\ell$. In Section~3, we prove that theorem and make some additional comments.

The author is grateful to A.~Dzhalilov for useful discussions.

\section{Construction}

Let us fix an arbitrary number $a>1$
and take an arbitrary increasing sequence of positive integers
$k_1,k_2,\dots,k_n,\dots$, where we assume $k_1\ge2a$, $k_1\ge5$, and $\sum_{n=1}^{+\infty}k_n^{-1}<+\infty$.
The irrational number given by the continued fraction
$$
\rho=[k_1,k_2,\ldots,k_n,\ldots]=1/(k_1+1/(k_2+1/(\dots/(k_n+\dots)))),
$$
which is defined as a limit of the sequence of rational convergents $p_n/q_n=[k_1,k_2,\dots,k_n]=1/(k_1+1/(k_2+1/(\dots/k_n)))$, will be the rotation number of our circle homeomorphism. It is known that the mutually
prime positive integers $p_n$ and $q_n$ satisfy the recurrent relation
$$
p_n=k_np_{n-1}+p_{n-2},\quad q_n=k_nq_{n-1}+q_{n-2},\quad n\ge1,
$$
where it is convenient to define $p_0=0$, $q_0=1$ and $p_{-1}=1$,
$q_{-1}=0$. Accordingly, the decreasing sequence of quantities $\Delta_n=|q_n\rho-p_n|=(-1)^n(q_n\rho-p_n)$ satisfies the relation
$$
\Delta_{n}=k_{n+2}\Delta_{n+1}+\Delta_{n+2},\quad n\ge-1.
$$

Let us also fix an arbitrary number $a>1$.

\begin{lemma}\label{lemma:sequence}
There exists a sequence of positive numbers $d_n$, $n\ge-1$, with $d_{-1}=1$, which satisfy the following relations:
\begin{equation}
d_{2n}=k_{2n+2}d_{2n+1}+d_{2n+2},\quad d_{2n-1}=k_{2n+1}d_{2n}+d_{2n+1}+(a-1)\delta_{2n},\label{d_n}
\end{equation}
where
\begin{equation}
\delta_{2n}=d_{2n+2}+a\delta_{2n+4}=\sum_{s=0}^{+\infty}a^{s}d_{2n+4s+2}<+\infty;\label{delta_n}
\end{equation}
and a number $\alpha>1$ such that the following inequality holds:
\begin{equation}
0\le\log\frac{\alpha d_n}{\Delta_n}<\frac{2(a-1)}{k_{n+2}}.\label{eq:d_n/Delta_n}
\end{equation}
%Also, we have $\delta_{2n}+2\delta_{2n+2}<d_{2n}$.
\end{lemma}

\begin{proof}
Let us put $d_n^{(0)}=\Delta_n$, $n\ge-1$, and define recurrently $d_{n}^{(r+1)}=k_{n+2}d_{n+1}^{(r)}+d_{n+2}^{(r)}+(a-1)\delta_{n}^{(r)}$ for $n\ge-1$, $r\ge0$, with $\delta_{n}^{(r)}=\sum_{s=0}^{+\infty}a^{s}d_{n+4s+2}^{(r)}$ for $n$ even, $\delta_{n}^{(r)}=0$ for $n$ odd. It is obvious that $d_n^{(r)}$ and $\delta_n^{(r)}$ are non-decreasing w.r.t. $r$ for any fixed $n$.

At the moment we cannot state that all defined quantities are finite, but it will be proven shortly. Since we have $\Delta_n=k_{n+2}\Delta_{n+1}+\Delta_{n+2}=(k_{n+2}k_{n+3}+1)\Delta_{n+2}+k_{n+2}\Delta_{n+3}>k_{n+2}k_{n+3}\Delta_{n+2}$, $n\ge-1$, it follows that $\frac{\Delta_{n+4s+2}}{\Delta_n}<\frac{1}{k_{n+2}\dots k_{n+4s+3}}$, $s\ge0$, and
\begin{equation}\label{eq:delta_n/Delta_n}
\frac{\delta_{n}^{(0)}}{\Delta_{n}}<\sum_{s=0}^{+\infty}\frac{a^s}{k_{n+2}\dots k_{n+3+4s}}<\frac{2}{k_{n+2}k_{n+3}}
\end{equation}
(here we used the assumed lower bound $2a$ for all $k_n$).

It is easy to see from the definition that (as soon as $d_n^{(r)}<+\infty$ for all $n\ge-1$)
$$
\frac{d_n^{(r+1)}}{d_n^{(r)}}\le\sup_{s\ge n+1}\frac{d_s^{(r)}}{d_s^{(r-1)}}\le\dots\le\sup_{s\ge n+r}\frac{d_s^{(1)}}{d_s^{(0)}}.
$$
Since $d_s^{(1)}=\Delta_s+(a-1)\delta_s^{(0)}$, $d_s^{(0)}=\Delta_s$, and $(k_n)_n$ is strictly increasing, the two latter displayed inequalities give us an estimate
$$
\frac{d_n^{(r+1)}}{d_n^{(r)}}<1+\frac{2(a-1)}{k_{n+r+2}k_{n+r+3}}.
$$
It follows that
$$
\log\frac{d_n^{(r+1)}}{\Delta_n}=\sum_{s=0}^r\log\frac{d_n^{(s+1)}}{d_n^{(s)}}<\sum_{s=0}^r\frac{2(a-1)}{k_{n+s+2}k_{n+s+3}}
<2\sum_{s=0}^r\left(\frac{a-1}{k_{n+s+2}}-\frac{a-1}{k_{n+s+3}}\right)<\frac{2(a-1)}{k_{n+2}}.
$$
In particular, all $d_n^{(r)}$ and $\delta_n^{(r)}$ are finite. Moreover, the whole set $\{d_n^{(r)}\}_{n,r}$ is bounded. Therefore there exist finite limits $d_n^*=\lim_{r\to+\infty}d_n^{(r)}$, $\delta_n^*=\lim_{r\to+\infty}\delta_n^{(r)}$, $n\ge-1$, that satisfy
$$
d_{n}^*=k_{n+2}d_{n+1}^*+d_{n+2}^*+(a-1)\delta_{n}^*,\quad\delta_{2n}^*=\sum_{s=0}^{+\infty}a^{s}d_{2n+4s+2}^*.
$$
By scaling them all down by putting $d_n=d_n^*/d_{-1}^*$, $n\ge-1$, we obtain the sequence that satisfies the conditions of the lemma with $\alpha=d_{-1}^*$.
%(It is easy to see from the proof that the inequality (\ref{eq:d_n/Delta_n}) is strict indeed, and it is a rough estimate anyway.)
\end{proof}

We assume $\delta_n=0$ for $n$ odd.

Notice that (\ref{d_n}) and (\ref{delta_n}) imply the inequalities:
\begin{equation}\label{eq:bounds on d_n and delta_n}
\frac{d_{n+2}}{d_{n}}<\frac{1}{k_{n+2}k_{n+3}},\qquad\frac{\delta_{2n}}{d_{2n}}<\frac{2}{k_{2n+2}k_{2n+3}}
\end{equation}

Now let us produce our main example of a piecewise linear circle homeomorphism $T$ through its lift function $f_0:\R\to\R$, $f_0(x+1)\equiv f_0(x)+1$, given by
$$
f_0(x)=\left\{\begin{array}{l}
ax+d_0,\quad0\le x\le \delta_0;\\
x+d_0+(a-1)\delta_0,\quad\delta_0\le x\le d_0+a\delta_2;\\
a^{-1}x+(2-a^{-1})d_0+(a-1)\delta_0+(a-1)\delta_2,\quad d_0+a\delta_2\le x\le d_0+a\delta_0+a\delta_2;\\
x+d_0,\quad d_0+a\delta_0+a\delta_2\le x\le1.
\end{array}\right.
$$
(It easily follows from (\ref{eq:bounds on d_n and delta_n}) that $d_0+a\delta_0+a\delta_2<1$.)
As one can see, $f_0(1)=1+d_0=f_0(0)+1$ indeed, and thus defined $T$ has four breaks: two of size $a$ and two of size $a^{-1}$. In the next section we will show that the rotation number of $T$ is $\rho$, its break points lie on four distinct trajectories, and the invariant measure of $T$ is absolutely continuous w.\ r.\ t.\ the Lebesgue measure $\ell$.

\begin{theorem}\label{theorem:main}
The circle diffeomorphism with breaks $T$ defined above has rotation number $\rho$, its four break points lie on pairwise distinct trajectories, and its invariant measure $\mu$ is absolutely continuous w.~r.~t.\ the Lebesgue measure $\ell$.
\end{theorem}

In what follows, we will represent a continuous piecewise linear function defined on a segment by indicating the values of its derivative over those subsegments where it differs from~1, rather than expressions for the function itself, as it makes formulas much more intuitive. (Such information together with a value at a single point lets one restore the whole function.) For example:
$$
f_0'(x)=\left\{\begin{array}{l}
a,\quad x\in(0,\delta_0);\\
a^{-1},\quad x\in(d_0+a\delta_2,d_0+a\delta_0+a\delta_2);\\
1\quad\text{elsewhere on [0,1]},
\end{array}\right.
$$
and this information together with the value $f_0(0)=d_0$ determines $f_0$ on $[0,1]$.

\section{Proof Of The Claim}

\subsection{Renormalization Basics}

For analyzing metrical properties of the dynamical system induced on $\unitcircle$ by the homeomorphism $T$ we will apply the renormalization approach, described in detail, for ex., in \cite{KT-cmp2013} as follows.

Given a circle homeomorphism $T$ with irrational rotation number $\rho$, one may `mark' a point $\xi_0\in\unitcircle$, consider its trajectory $\xi_i=\xi_i(\xi_0)=T^i\xi_0\in\unitcircle$, $i\in\Z$, and pick out of it the sequence of the `dynamical convergents' $\{\xi_{q_n}\}_n$ indexed by the denominators of the consecutive rational
convergents to $\rho$. The dynamical convergents approach the marked point $\xi_0$, alternating their order in the following way:
$$
\xi_{q_{\!-1}}<\xi_{q_1}<\xi_{q_3}<\dots<\xi_{q_{2m+1}}<\dots<\xi_0<\dots<\xi_{q_{2m}}<\dots<\xi_{q_2}<\xi_{q_0}
$$
(the point $\xi_{q_{-\!1}}$ on the circle coincides with $\xi_0$; here it may be seen as $\xi_{q_{-\!1}}=\xi_0-1$). Define the $n$th `fundamental segment'
$\Delta_0^{(n)}=\Delta_0^{(n)}(\xi_0)$ as the circle arc $[\xi_0,\xi_{q_n}]$ if $n$ is even and $[\xi_{q_n},\xi_0]$ if $n$ is odd.
The iterates $T^{q_{n}}$ and $T^{q_{n-1}}$ restricted to $\Delta_0^{(n-1)}$ and $\Delta_0^{(n)}$ respectively are
two continuous components of the first-return map for $T$ on the segment
$\overline\Delta_0^{(n-1)}=\Delta_0^{(n-1)}\cup\Delta_0^{(n)}$.
The consecutive images of $\Delta_0^{(n-1)}$ and
$\Delta_0^{(n)}$ until the return to $\overline\Delta_0^{(n-1)}$ cover the whole circle without overlapping
beyond their endpoints, thus forming the $n$th `dynamical partition'
$\Pdyn_n(\xi_0)=\{\Delta_i^{(n-1)},0\le i<q_{n}\}\cup\{\Delta_i^{(n)},0\le i<q_{n-1}\}$ of $\unitcircle$,
where $\Delta_i^{(n)}=\Delta_i^{(n)}(\xi_0)$ stands for $T^i\Delta_0^{(n)}$.

For given $n\ge0$, one may consider the pair of
functions $(f_{n,\xi_0},g_{n,\xi_0})$ called sometimes the $n$-th `pre-renormalization' of $T$, which are the mappings $T^{q_{n}}$ and $T^{q_{n-1}}$ restricted to $\Delta_0^{(n-1)}$
and $\Delta_0^{(n)}$ respectively, in the coordinate system with the origin at $\xi_0$:
$$
f_{n,\xi_0}=\pi_{\xi_0}\circ T^{q_{n}}\circ\pi_{\xi_0}^{-1},\qquad g_{n,\xi_0}=\pi_{\xi_0}\circ T^{q_{n-1}}\circ\pi_{\xi_0}^{-1},
$$
where $\pi_{\xi_0}$ is a linear mapping in a neighbourhood of $\xi_0$ that sends $\xi_0$ to $0$ and changes neither length no orientation. We will omit the index $\xi_0$, when it is clear which point is marked. Both $f_n$ and $g_n$ are strictly increasing continuous functions defined on the segments $[-\ell(\Delta_0^{(n-1)}),0]$ and $[0,\ell(\Delta_0^{(n)})]$ for $n$ even and on $[0,\ell(\Delta_0^{(n-1)})]$ and $[-\ell(\Delta_0^{(n)}),0]$ for $n$ odd respectively. Since $f_n(0)=(-1)^n\ell(\Delta_0^{(n)})$, $g_n(0)=(-1)^{n-1}\ell(\Delta_0^{(n-1)})$ and $f_n(g_n(0))=g_n(f_n(0))$, the functions $f_n$ and $g_n$ are said to form a `commuting pair'. (Notice that our pairs $(f_n,g_n)$ differ from those in~\cite{KT-cmp2013} by having not been rescaled.)

This definition is slightly ambiguous for $n=0$, in which case we define $f_0$ on the segment $[-1,0]$ representing the unit circle $\unitcircle$ cut by the marked point $\xi_0$ (one can write $\Delta_0^{(-1)}=[\xi_0-1,\xi_0]$), so that $f_0$ is a lift of $T$ onto $\R$, shifted to the marked point, while $g_0=\Id-1$ defined on $[0,f_0(0)]$.

Notice that $g_{n+1}$ is $f_n$ restricted to a smaller segment, while $f_{n+1}=f_{n}^{k_{n+1}}\circ g_{n}$, and the `renormalization height' $k_{n+1}$ is the largest integer $k$ such that $f_n^k(g_n(0))$ has the same sign as $g_n(0)$ (it is also a partial quotient in the continued fraction $\rho=[k_1,k_2,\dots,k_n,\dots]$), $n\ge0$.

In this paper, we shall also consider the `backward dynamical partition segments' $\nabla_i^{(n)}=T^{-i}\nabla_0^{(n)}$, where $\nabla_0^{(n)}=\nabla_0^{(n)}(\xi_0)$ is the circle arc $[\xi_0,\xi_{-q_n}]$ if $n$ is odd and $[\xi_{-q_n},\xi_0]$ if $n$ is even, and define the `backward pre-renormalization' as the pair of functions
$$
\cf_{n}=\pi_{\xi_0}\circ T^{-q_{n}}\circ\pi_{\xi_0}^{-1},\qquad \cg_{n}=\pi_{\xi_0}\circ T^{-q_{n-1}}\circ\pi_{\xi_0}^{-1}
$$
restricted to $\nabla_0^{(n-1)}$ and $\nabla_0^{(n)}$ respectively. These functions have properties similar to $f_n$ and $g_n$, in particular $\cf_{n+1}=\cf_{n}^{k_{n+1}}\circ\cg_{n}$.

\subsection{Measurements}
\label{subsection:measurements}

Now, let us turn to the specific homeomorphism $T$ that was defined at the end of Section~2 (notice that the meaning of $f_0$ there is consistent with Subsection~3.1, though the expressions were written for the segment $[0,1]$ rather than $[-1,0]$).
%For $n\ge1$ there is no sense to extend $f_n$ and $g_n$ to whole $\R$.

\begin{proposition}\label{prop:g_n,f_n}
For every $n\ge0$, we have $\ell(\Delta_0^{(n)})=d_n$, and for every $n\ge1$

$$
g_{2n-1}'(x)=1\text{ everywhere on }[-d_{2n-1},0],
$$

$$%\begin{equation}
f_{2n-1}'(x)=\left\{\begin{array}{l}
a,\quad x\in(0,\delta_{2n});\\
a^{-1},\quad x\in(\delta_{2n-2},\delta_{2n-2}+a\delta_{2n});\\
1\text{ elsewhere on }[0,d_{2n-2}],
\end{array}\right.
$$%\end{equation}

$$%\begin{equation}
f_{2n}'(x)=\left\{\begin{array}{l}
a^{-1},\quad x\in(-d_{2n-1}+a\delta_{2n+2},-d_{2n-1}+a\delta_{2n+2}+a\delta_{2n});\\
1\text{ elsewhere on }[-d_{2n-1},0],
\end{array}\right.
$$%\end{equation}

$$%\begin{equation}
g_{2n}'(x)=\left\{\begin{array}{l}
a,\quad x\in(0,\delta_{2n});\\
1\text{ elsewhere on }[0,d_{2n}].
\end{array}\right.
$$%\end{equation}

The rotation number of $T$ is $\rho=[k_1,k_2,\dots,k_n,\dots]$.
\end{proposition}

\begin{remark}
When we state that a formula like those above holds, it is a part of our statement that the segments in it are positioned correctly (for example, formula for $f_{2n-1}$ includes the claim that the segments $(0,\delta_{2n})$ and $(\delta_{2n-2},\delta_{2n-2}+a\delta_{2n})$ do not overlap and are both contained within the segment $[0,d_{2n-2}]$). Such correctness is usually easy to check due to (\ref{eq:bounds on d_n and delta_n}) and the assumed lower bounds on $k_n$.
\end{remark}

\begin{proof}
On each induction step $m\ge1$, one must follow the sequence of points
$$
g_{m-1}(0),f_{m-1}(g_{m-1}(0)),\dots,f_{m-1}^k(g_{m-1}(0)),\dots
$$
until the moment right before it jumps over zero, and each moment $k$ calculate the function $f_{m-1}^k\circ g_{m-1}$
%as it maps the segment $[0,f_m(0)]$ onto  $[f_m^{k-1}(g_m(0)),f_m^k(g_m(0))]$ for $m$ even, or $[f_m(0),0]$ onto $[f_m^{k}(g_m(0)),f_m^{k-1}(g_m(0))]$ for $m$ odd. This calculation is done
by composing a previously calculated $f_{m-1}^{k-1}\circ g_{m-1}$ with $f_{m-1}$ on an appropriate segment. The last moment will be the renormalization height $k_{m}$ (a partial quotient in the continued fraction expansion for the rotation number of $T$), while the function $f_{m-1}^{k_m}\circ g_{m-1}$ will be the new $f_m$.

For $m=1$, the segment $[g_{0}(0),f_{0}(g_{0}(0))]$ has length $d_0$. The length of the segment $[f_{0}(g_{0}(0)),f_{0}^{2}(g_{0}(0))]$ is $d_0+(a-1)\delta_0$ due to stretching by the `upper tooth' of derivative $f_0'(x)=a$ on $(-1,-1+\delta_0)\subset[g_{0}(0),f_{0}(g_{0}(0))]=[-1,-1+d_0]$. The length of $[f_{0}^{k-1}(g_{0}(0)),f_{0}^{k}(g_{0}(0))]$ for $k>2$ is $d_0$ again due shrinking back by the `lower tooth' of derivative $f_0'(x)=a^{-1}$ on $(-1+d_0+a\delta_2,-1+d_0+a\delta_0+a\delta_2)\subset[f_0(g_{0}(0)),f_{0}^2(g_{0}(0))]=[-1+d_0,-1+2d_0+\delta_0]$. The renormalization height is $k_{1}$, and $f_{1}(0)=f_{0}^{k_{1}}(g_{0}(0))=-d_{1}$, both due to (\ref{d_n}). The two mentioned teeth of non-unit derivative partially annihilate each other, and their leftovers appear as the appropriate teeth in new $f_1=f_0^{k_1}\circ g_0$ on $[0,d_0]$.

For $m=2n$, $n\ge1$, the lengths of all segments $[f_{2n-1}^{k}(g_{2n-1}(0)),f_{2n-1}^{k-1}(g_{2n-1}(0))]$ is equal to $d_{2n-1}$. The renormalization height is $k_{2n}$, and $f_{2n}(0)=f_{2n-1}^{k_{2n}}(g_{2n-1}(0))=d_{2n}$, both due to (\ref{d_n}). The only `tooth' of non-unit derivative, which is $f_{2n-1}'(x)=a^{-1}$ on $(\delta_{2n-2},\delta_{2n-2}+a\delta_{2n})\subset[d_{2n},d_{2n}+d_{2n-1}]=[f_{2n-1}^{k_{2n}}(g_{2n-1}(0)),f_{2n-1}^{k_{2n}-1}(g_{2n-1}(0))]$, appears at the last moment $k=k_{2n}$, and is being transferred to the segment $[-d_{2n-1},0]$ where new $f_{2n}$ is defined. It will be exactly the tooth $f_{2n}'(x)=a^{-1}$ on $(-d_{2n-1}+a\delta_{2n+2},-d_{2n-1}+a\delta_{2n+2}+a\delta_{2n})$ due to the relation $\delta_{2n-2}-d_{2n}=a\delta_{2n+2}$ implied by (\ref{delta_n}).

For $m=2n+1$, $n\ge1$, it is similar to the case of $m=1$: all the segments $[f_{2n}^{k-1}(g_{2n}(0)),f_{2n}^{k}(g_{2n}(0))]$, $k\ge1$, but one have lengths $d_{2n}$ with the only difference that the longer segment of length $d_{2n}+(a-1)\delta_{2n}$ we have for $k=1$ rather than for $k=2$. The upper tooth $g_{2n}'(x)=a$ on $(0,\delta_{2n})$ interacts with the lower tooth $f_{2n}'=a^{-1}$ on $(-d_{2n-1}+a\delta_{2n+2},-d_{2n-1}+a\delta_{2n+2}+a\delta_{2n})\subset[g_{2n}(0),f_{2n}(g_{2n}(0))]=[-d_{2n-1},-d_{2n-1}+d_{2n}+(a-1)\delta_{2n}]$ giving birth to the appropriate teeth in new $f_{2n+1}$. The renormalization height is $k_{2n+1}$, and $f_{2n+1}(0)=f_{2n}^{k_{2n+1}}(g_{2n}(0))=-d_{2n+1}$, both due to (\ref{d_n}).
\end{proof}

\begin{proposition}\label{prop:cf_n,cg_n}
For every $n\ge0$, we have $\ell(\nabla_0^{(n)})=d_n$,

$$%\begin{equation}
\cg_{2n}'(x)=1\text{ everywhere on }[-d_{2n},0],
$$%\end{equation}

$$%\begin{equation}
\cf_{2n}'(x)=\left\{\begin{array}{l}
a^{-1},\quad x\in(d_{2n},d_{2n}+a\delta_{2n});\\
a,\quad x\in(2d_{2n}+(a-1)\delta_{2n}+a\delta_{2n+2},2d_{2n}+a\delta_{2n}+a\delta_{2n+2});\\
1\text{ elsewhere on }[0,d_{2n-1}],
\end{array}\right.
$$%\end{equation}

$$%\begin{equation}
\cf_{2n+1}'(x)=\left\{\begin{array}{l}
a^{-1},\quad x\in(-d_{2n+1},-d_{2n+1}+a\delta_{2n+2});\\
a,\quad x\in(-d_{2n+1}+a\delta_{2n+2}+\delta_{2n}-\delta_{2n+2},-d_{2n+1}+a\delta_{2n+2}+\delta_{2n});\\
1\text{ elsewhere on }[-d_{2n},0],
\end{array}\right.
$$%\end{equation}

$$
\cg_{2n+1}'(x)=1\text{ everywhere on }[0,d_{2n+1}].
$$
\end{proposition}

\begin{proof}
First we calculate $\cf_0$ as $f_0^{-1}$ restricted to $[0,1]$, and then analyze consecutive iterates similarly to the proof of Proposition~\ref{prop:g_n,f_n}.
\end{proof}

Define the sequence of sets $\Theta_n=\bigcup_{i\in N_n}\nabla_i^{(n)}\subset\unitcircle$, $n\ge0$, formed of backward dynamical partition segments, inductively as follows: let $N_0=\{0\le i\le k_1-3\}$ and $N_{n}=\{i+q_{n-1}+jq_{n}:i\in N_{n-1}, 0\le j\le k_{n+1}-3\}$, $n\ge1$. It is easy to see that $N_n\subset\{0\le i<q_{n+1}\}$ and $\Theta_{n+1}\subset\Theta_{n}$, $n\ge0$. Saying it in words, we consecutively cut away little pieces of the circle nearby every preimage of the break point $\xi_0$.

\begin{lemma}\label{lemma:nabla_i^n}
For every $i\in N_n$, $n\ge0$, we have $\ell(\nabla_i^{(n)})=d_n$.
\end{lemma}
\begin{proof}
The definition of $N_n$ implies that for every $i\in N_n$ there exist integers $0\le j_{m+1}\le k_{m+1}-3$, $0\le m\le n$, such that $i=\sum_{m=0}^{n}(q_{m-1}+j_{m+1}q_m)$. Accordingly, $\nabla_i^{(n)}=(\cf_0^{j_1}\circ\cg_0)\circ\dots\circ(\cf_n^{j_{n+1}}\circ\cg_n)\nabla_0^{(n)}$. In view of Proposition~\ref{prop:cf_n,cg_n}, it is easy to check that consecutive images of $\nabla_0^{(n)}$ in the latter sequence never fall onto the teeth of non-unit derivative of $\cf_m$ and $\cg_m$.
\end{proof}

\begin{proposition}\label{prop:Theta}
The limit set $\Theta=\bigcap_{n\ge0}\Theta_n$ has positive Lebesgue measure.
\end{proposition}
\begin{proof}
It follows from Lemma~\ref{lemma:nabla_i^n} that $\ell(\Theta_n)=d_n\cdot\#(N_n)$, and by definition of $N_n$ we have $\#(N_n)=(k_{n+1}-2)\cdot\#(N_{n-1})$. Hence, $\ell(\Theta_n)=\frac{(k_{n+1}-2)d_n}{d_{n-1}}\ell(\Theta_{n-1})$. The relations (\ref{d_n}), (\ref{eq:bounds on d_n and delta_n}) and increasing of $k_n$ imply that $d_{n-1}<(k_{n+1}+1)d_n$ and therefore $\ell(\Theta_n)>\bigl(1-\frac{3}{k_{n+1}+1}\bigr)\ell(\Theta_{n-1})>\exp\bigl(-\frac{6}{k_{n+1}}\bigr)\ell(\Theta_{n-1})$ for large enough $n$. The statement follows from assumed convergence of the series $\sum_nk_n^{-1}$.
\end{proof}

\begin{proposition}\label{prop:f_n,eta,g_n,eta}
Let $\eta_0\in\Theta$. Then for every $n\ge0$, we have $\ell(\Delta_0^{(n)}(\eta_0))=d_n$,

$$%\begin{equation}
f_{2n,\eta_0}'(x)=\left\{\begin{array}{l}
a,\quad x\in(b_{2n},b_{2n}+\delta_{2n});\\
a^{-1},\quad x\in(b_{2n}+d_{2n}+a\delta_{2n+2},b_{2n}+d_{2n}+a\delta_{2n}+a\delta_{2n+2});\\
1\text{ elsewhere on }[-d_{2n-1},0],
\end{array}\right.
$$%\end{equation}

$$%\begin{equation}
g_{2n,\eta_0}'(x)=1\text{ everywhere on }[0,d_{2n}],
$$%\end{equation}

$$
g_{2n+1,\eta_0}'(x)=1\text{ everywhere on }[-d_{2n+1},0],
$$

$$%\begin{equation}
f_{2n+1,\eta_0}'(x)=\left\{\begin{array}{l}
a,\quad x\in(b_{2n+1},b_{2n+1}+\delta_{2n+2});\\
a^{-1},\quad x\in(b_{2n+1}+\delta_{2n},b_{2n+1}+\delta_{2n}+a\delta_{2n+2});\\
1\text{ elsewhere on }[0,d_{2n}],
\end{array}\right.
$$%\end{equation}

\noindent with certain values $b_{2n}\in(-d_{2n-1},f_{2n}^{k_{2n+1}-2}(-d_{2n-1})]$, $b_{2n+1}\in[f_{2n+1}^{k_{2n+2}-2}(d_{2n}),d_{2n})$.
\end{proposition}

\begin{proof}
It is easy to follow the sequence of iterates like we did in the proof of Proposition~\ref{prop:g_n,f_n}, this time using $\eta_0$ as the marked point. What is important here are the bounds on values $(b_n)_{n\ge0}$ that follow directly from our construction of the set $\Theta$. Indeed, $b_n$ is nothing else but a coordinate of an appropriate preimage of the break point $\xi_0$ in vicinity of $\eta_0$. Since passing from $\Theta_{n-1}$ to $\Theta_n$ we have cut away from every $\nabla_i^{(n-1)}$, $i\in N_{n-1}$, three adjoint segments $\nabla_{i+q_{n-1}+(k_{n+1}-2)q_{n}}^{(n)}$, $\nabla_{i+q_{n-1}+(k_{n+1}-1)q_{n}}^{(n)}$, and $\nabla_i^{(n+1)}$, then $b_{2n}\not\in(f_{2n}^{k_{2n+1}-2}(-d_{2n-1}),0)$ and $b_{2n+1}\not\in(0,f_{2n+1}^{k_{2n+2}-2}(d_{2n}))$.
\end{proof}

\subsection{Differentiability}

In what follows, all the renormalization structures are counted from $\eta_0$ as the marked point, and we omit the letter $\eta_0$ by $f_n$, $g_n$, $\Pdyn_n$ and $\Delta_i^{(n)}$.

\begin{proposition}\label{prop:differentiability}
The conjugacy $\phi$ from (\ref{eq:conjugacy}) is differentiable at every point $\eta_0\in\Theta$.
\end{proposition}

First we need to prove the following lemma about the lengths of segments from dynamical partitions of several deeper levels lying within the fundamental segments. Denote by $\Gamma_{n,s}$ the set of all segments $\Delta_i^{(n+s-1)}\in\Pdyn_{n+s}$, which lie within the segment $\Delta_0^{(n-1)}$.

\begin{lemma}\label{lemma:Gamma}
$\frac{\ell(\Delta_i^{(n+s-1)})}{d_{n+s-1}}\to 1$ as $n\to+\infty$ uniformly in $i$ for all $\Delta_i^{(n+s-1)}\in\Gamma_{n,s}$, $0\le s\le2$, and for all those $\Delta_i^{(n+2)}\in\Gamma_{n,3}$, which lie within $\Delta_{q_{n+1}-q_{n-1}}^{(n)}$.
\end{lemma}

\begin{proof}
In this proof, for simplicity we identify a set $M$ on the circle in vicinity of $\eta_0$ with its projection $\pi_{\eta_0}M$ to the real line, where $f_n$ and $g_n$ operate.

For $s=0$ the statement is trivial, since $\Gamma_{n,0}=\{\Delta_0^{(n-1)}\}$, and $\ell(\Delta_0^{(n-1)})=d_{n-1}$ by Proposition~\ref{prop:f_n,eta,g_n,eta}.

It is easy to see that $\Gamma_{n,s}=\Gamma_{n+2,s-2}\cup\bigcup_{j=0}^{k_{n+1}-1}f_n^j\circ g_n(\Gamma_{n+1,s-1})$. We will use this recurrent relation in order to estimate the lengths of $\Delta_i^{(n+s-1)}\in\Gamma_{n,s}$ for $s\ge1$. As seen from Proposition~\ref{prop:f_n,eta,g_n,eta}, $g_n$ does not distort length of any segment, while $f_n$ may once stretch and once shrink it, with the maximal possible change in length equal to $(a-1)\delta_{n}$ for $n$ even and $(a-1)\delta_{n+1}$ for $n$ odd.

For $s=1$, the segments in $\Gamma_{n,1}$ are the images of $\Delta_0^{(n)}$ under application of $g_n$ (once) and then $f_n$ (from~0 to $k_{n+1}-1$ times). Using Proposition~\ref{prop:f_n,eta,g_n,eta}, one can easily check that all of them have length $d_n$ for $n$ odd, while for $n$ even all but one have length $d_{n}$ with that exceptional one of length $d_{n}+(a-1)\delta_{n}$.

For $s=2$, the segments in $\Gamma_{n,2}\backslash\Gamma_{n+2,0}$ are the images of the segments in $\Gamma_{n+1,1}$ under $f_n^j\circ g_n$, $0\le j<k_{n+1}$. Proposition~\ref{prop:f_n,eta,g_n,eta} allows one to calculate that their lengths are either $d_{n+1}$ or $d_{n+1}+(a-1)\delta_{n}$ for $n$ even, and either $d_{n+1}$ or $d_{n+1}+(a-1)\delta_{n+1}$ for $n$ odd.

For $s=3$, the segments in $\Gamma_{n,3}\backslash\Gamma_{n+2,1}$ are the images of the segments in $\Gamma_{n+1,2}$ under $f_n^j\circ g_n$, $0\le j<k_{n+1}$. If $n$ is odd, then
their lengths differ from $d_{n+2}$ by not more than $2(a-1)\delta_{n+1}$. If $n$ is even, then this difference formally can reach $(a-1)\delta_{n+2}+(a-1)\delta_{n}$, but let us look closer at this case. Notice that for $n$ even the teeth of non-unit derivative in $f_{n}$ are positioned in such a way that the maximal possible change in length of a segment under their {\em cumulative} action (i.e. the action of $f_{n}^2$ on $\Delta_i^{(n)}\in\Pdyn_{n+1}$ such that $b_{n}\in\Delta_i^{(n)}$) is $(a-1)\delta_{n+2}$ rather than $(a-1)\delta_{n}$. In fact, there are exactly two segments in $\Gamma_{n,3}$ that are possibly stretched by more than $2(a-1)\delta_{n+2}$: those whose preimages under $f_n$ intersect the upper tooth $(b_{2n},b_{2n}+\delta_{2n})$ (while they themselves intersect the lower tooth and are being shrank back on the next application of $f_n$). It also follows from Proposition~\ref{prop:f_n,eta,g_n,eta} that these two exceptional segments do not intersect $[f_{n}^{k_{n+1}-1}(-d_{n-1}),0]=\Delta_{q_{n+1}-q_{n-1}}^{(n)}$.

Since $\delta_n=o(d_{n+1})$ (in fact $\delta_{2n}/d_{2n+2}\to1$ as $n\to+\infty$), the estimates we proved imply the statement of the lemma.
\end{proof}

\begin{proof}[Proof of Proposition~\ref{prop:differentiability}]
We will show directly that $\frac{\ell([\phi(\eta),\phi(\eta_0)])}{\ell([\eta,\eta_0])}\to\alpha$ as $\eta\to\eta_0$, and therefore $\phi'(\eta_0)=\alpha$ for every $\eta_0\in\Theta$ (here and in the sequel, by $[\xi,\zeta]$ we denote the smallest of the two arcs $[\xi,\zeta]$ and $[\zeta,\xi]$ on $\unitcircle$).

As $\eta$ tends to $\eta_0$ from the left (from the right), it passes through the sequence of segments $\dot\Delta_0^{(n-1)}=\Delta_0^{(n-1)}\backslash\Delta_0^{(n+1)}\cup\{\eta_{q_{n+1}}\}$ with $n$ even (odd), which in turn consist of segments from the dynamical partitions of the next levels. According to the decomposition $\dot\Delta_0^{(n-1)}=\bigcup_{j_0=0}^{k_{n+1}-1}(\Delta_{q_{n-1}+(k_{n+1}-1-j_0)q_n}^{(n+2)}\cup\bigcup_{j_1=0}^{k_{n+2}-1}\Delta_{q_{n-1}+(k_{n+1}-j_0)q_n+j_1q_{n+1}}^{(n+1)})$, consider three cases.

\underline {Case~1:} $\eta\in\Delta_{q_{n-1}+(k_{n+1}-j_0)q_n+j_1q_{n+1}}^{(n+1)}$, $1\le j_0\le k_{n+1}-1$ (we have excluded $j_0=0$, which will be Case~3 below), $0\le j_1\le k_{n+2}-1$. On one hand, the segment $[\eta,\eta_0]$ contains $j_0$ segments from $\Gamma_{n,1}$ (namely $\Delta_{q_{n-1}+iq_n}^{(n)}$, $k_{n+1}-j_0\le i\le k_{n+1}-1$) and $j_1+1$ segments from $\Gamma_{n,2}$ (namely $\Delta_{q_{n-1}+(k_{n+1}-j_0)q_n+iq_{n+1}}^{(n+1)}$, $0\le i<j_1$, and $\Delta_0^{(n+1)}$). On another hand, it is contained in the union of $j_0$ segments from $\Gamma_{n,1}$ (the same as previously) and $j_1+2$ segments from $\Gamma_{n,2}$ (with $\Delta_{q_{n-1}+(k_{n+1}-j_0)q_n+j_1q_{n+1}}^{(n+1)}$ added). Due to Lemma~\ref{lemma:Gamma} and the limit $d_n/\Delta_n\to1/\alpha$ implied by Lemma~\ref{lemma:sequence}, the following estimate holds with some $\eps_n>0$, $\eps_n\to0$ as $n\to+\infty$:
$$
e^{-\eps_n}\alpha^{-1}(j_0\Delta_{n}+(j_1+1)\Delta_{n+1})\le\ell([\eta,\eta_0])\le\alpha^{-1}(j_0\Delta_{n}+(j_1+2)\Delta_{n+1})e^{\eps_n}.
$$
Since $\ell(\phi(\Delta_i^{(n)}))=\Delta_n$ for every $i,n\ge0$, we similarly obtain that
$$
j_0\Delta_{n}+(j_1+1)\Delta_{n+1}\le\ell([\phi(\eta),\phi(\eta_0)])\le j_0\Delta_{n}+(j_1+2)\Delta_{n+1}.
$$
Taking the ratio of the two latter estimates, we find it out that
\begin{multline*}
\left|\log\frac{\ell([\phi(\eta),\phi(\eta_0)])}{\alpha\ell([\eta,\eta_0])}\right|\le\eps_n+\log\frac{j_0\Delta_{n}+(j_1+2)\Delta_{n+1}}{j_0\Delta_{n}+(j_1+1)\Delta_{n+1}}\\
\le\eps_n+\frac{\Delta_{n+1}}{j_0\Delta_{n}+(j_1+1)\Delta_{n+1}}\le\eps_n+\frac{\Delta_{n+1}}{\Delta_{n}}\le\eps_n+\frac{1}{k_{n+2}}\to0,\ n\to+\infty.
\end{multline*}

\underline {Case~2:} $\eta\in\Delta_{q_{n-1}+(k_{n+1}-1-j_0)q_n}^{(n+2)}$, $0\le j_0\le k_{n+1}-1$. By similar consideration we get the estimates
\begin{gather*}
e^{-\eps_n}\alpha^{-1}(j_0\Delta_{n}+(k_{n+2}+1)\Delta_{n+1})\le\ell([\eta,\eta_0])\le\alpha^{-1}((j_0+1)\Delta_{n}+\Delta_{n+1})e^{\eps_n},\\
j_0\Delta_{n}+(k_{n+2}+1)\Delta_{n+1}\le\ell([\phi(\eta),\phi(\eta_0)])\le (j_0+1)\Delta_{n}+\Delta_{n+1},
\end{gather*}
with $\eps_n>0$, $\eps_n\to0$ as $n\to+\infty$, which imply
$$
\left|\log\frac{\ell([\phi(\eta),\phi(\eta_0)])}{\alpha\ell([\eta,\eta_0])}\right|\le\eps_n+\frac{\Delta_{n}-k_{n+2}\Delta_{n+1}}{j_0\Delta_{n}+(k_{n+2}+1)\Delta_{n+1}}\le
\eps_n+\frac{\Delta_{n+2}}{\Delta_{n}}\le\eps_n+\frac{1}{k_{n+2}k_{n+3}},
$$
with the same conclusion.

\underline {Case 3:} $\eta\in\Delta_{q_{n-1}+k_{n+1}q_n+j_1q_{n+1}}^{(n+1)}=\Delta_{(j_1+1)q_{n+1}}^{(n+1)}$, $0\le j_1\le k_{n+2}-1$. The segment $[\eta,\eta_0]$ contains $j_1+1$ segments from $\Gamma_{n,2}$ (namely $\Delta_{iq_{n+1}}^{(n+1)}$, $0\le i\le j_1$,) and some $0\le j_2\le k_{n+3}-1$ segments from $\Gamma_{n,3}$ (namely $\Delta_{(j_1+2)q_{n+1}+iq_{n+2}}^{(n+2)}$, $k_{n+3}-j_2\le i<k_{n+3}$). On the other hand, it is contained within the union of $j_1+1$ segments from $\Gamma_{n,2}$ (the same as previously), $j_2+1$ segments from $\Gamma_{n,3}$ (with added $\Delta_{(j_1+2)q_{n+1}+(k_{n+3}-j_2-1)q_{n+2}}^{(n+2)}$), and $\Delta_{(j_1+1)q_{n+1}}^{(n+3)}$ from $\Gamma_{n,4}$. Notice, that all the segments from $\Gamma_{n,3}$ mentioned here do lie within $\Delta_{q_{n+1}-q_{n-1}}^{(n)}$, so that the uniform limit in Lemma~\ref{lemma:Gamma} holds for them. For the tiny segment $\Delta_{(j_1+1)q_{n+1}}^{(n+3)}$ it will suffice to use a rough estimate $\ell(\Delta_{(j_1+1)q_{n+1}}^{(n+3)})\le a^2d_{n+3}$ that follows from the fact that it is an image of $\Delta_0^{(n+3)}$ under $f_n^{k_{n+1}-1}\circ g_n\circ f_{n+1}^{j_1}\circ g_{n+1}$, where only one step of applying $f_{n+1}$ and only one step of applying $f_n$ can stretch the length (and with the factor of not more than $a$) according to Proposition~\ref{prop:f_n,eta,g_n,eta}. We reach the estimates
\begin{multline*}
e^{-\eps_n}\alpha^{-1}((j_1+1)\Delta_{n+1}+j_2\Delta_{n+2})\le\ell([\eta,\eta_0])\\
\le\alpha^{-1}((j_1+1)\Delta_{n+1}+(j_2+1)\Delta_{n+2}+a^2\Delta_{n+3})e^{\eps_n}\\
\le\alpha^{-1}((j_1+1)\Delta_{n+1}+(j_2+1)\Delta_{n+2})(1+a^2/k_{n+2})e^{\eps_n},
\end{multline*}
with $\eps_n>0$, $\eps_n\to0$ as $n\to+\infty$, and
\begin{multline*}
(j_1+1)\Delta_{n+1}+j_2\Delta_{n+2}\le\ell([\phi(\eta),\phi(\eta_0)])\\
\le (j_1+1)\Delta_{n+1}+(j_2+1)\Delta_{n+2}+\Delta_{n+3},\\
\le ((j_1+1)\Delta_{n+1}+(j_2+1)\Delta_{n+2})(1+1/k_{n+2}).
\end{multline*}
Finally, we obtain
$$
\left|\log\frac{\ell([\phi(\eta),\phi(\eta_0)])}{\alpha\ell([\eta,\eta_0])}\right|\le
\eps_n+\frac{a^2}{k_{n+2}}+\frac{\Delta_{n+2}}{(j_1+1)\Delta_{n+1}+j_2\Delta_{n+2}}\le\eps_n+\frac{a^2}{k_{n+2}}+\frac{1}{k_{n+3}},
$$
which tends to zero as well.
\end{proof}

\subsection{Conclusions}

\begin{proof}[Proof of Theorem~\ref{theorem:main}] The theorem contains three statements, which we will prove now.

1. The rotation number of $T$ is $\rho$ by Proposition~\ref{prop:g_n,f_n}.

2. Four break points lie on pairwise distinct trajectories by Proposition~\ref{prop:f_n,eta,g_n,eta}. Indeed, by the construction, the break points of $f_n$ are nothing else but the projections of first preimages of break points of $T$ within the segment $\Delta_0^{(n-1)}$. If two of the latter points lied on the same trajectory, then for large enough $n$ their first preimages within $\Delta_0^{(n-1)}$ would coincide, which contradicts to the fact that $f_n$ has four distinct break point for any $n$.

3. For a circle diffeomorphism with breaks, its invariant measure $\mu$ can be either absolutely continuous, or singular w.~r.~t.\ $\ell$ (see, for ex., \cite{KH-book}, where this result is proved for pure diffeomorphisms, and is easily extended to those with breaks). If $\mu$ is singular, then the conjugacy $\phi$ is a singular function, and $\phi'=0$ almost everywhere~$\ell$, which contradicts to Propositions~\ref{prop:Theta} and~\ref{prop:differentiability}. Therefore, $\mu$ is absolutely continuous.
\end{proof}

\begin{remark}
The linearizing change of variables $\phi$ in (\ref{eq:conjugacy}) is absolutely continuous and therefore possesses a density $h\in L_1(\unitcircle)$. This density is positive a.~e.~$\ell$ due to the fact that $T$ is ergodic w.~r.~t.~$\ell$ (see \cite{KH-book}).
\end{remark}

\begin{remark}
The conjugacy $\phi$ is {\em essentially} absolutely continuous, i.~e.\ it is not piecewise $C^1$, moreover, there is an everywhere dense subset of $\unitcircle$, over which $\phi$ is not differentiable.
\end{remark}

\begin{remark}
An irrational number $\rho$ is said to belong to Diophantine class $D_\delta$ if there exists a constant $C>0$
such that $|\rho-p/q|\ge Cq^{-2-\delta}$ for any rational number $p/q$. The set of Roth numbers is $R=\cap_{\delta>0}D_\delta$. The numbers from $D_0$ are those whose continued fraction expansion is formed by a bounded sequence $(k_n)_n$, therefore they are called `numbers of bounded type'. The restrictions we have imposed on $(k_n)_n$ allow one to choose $\rho\in R$ (for ex., $k_n=(n+5)^2$ will work), but not $\rho\in D_0$.
\end{remark}

\end{document}